\documentclass[12pt,a4paper]{amsart}
\usepackage[utf8]{inputenc}
\usepackage[top=35mm, bottom=30mm, left=30mm, right=30mm]{geometry}
\usepackage{mathptmx}
\usepackage{mathrsfs}

\usepackage{verbatim}
\usepackage{url}
\usepackage[all]{xy}
\usepackage{xcolor}
\usepackage[colorlinks=true,citecolor=blue]{hyperref}

\usepackage{amsmath,amsthm}
\usepackage{amssymb}

\usepackage{enumerate}

\usepackage{graphicx}

\newtheorem{thm}{Theorem}[section]

\newtheorem{lem}[thm]{Lemma}
\newtheorem{prop}[thm]{Proposition}

\newtheorem{ques}[thm]{Question}

\theoremstyle{definition}

\numberwithin{equation}{section}

\begin{document}

%%%%% To ease editing, for IMPAN journals add:

\baselineskip=17pt

%%%%%%%%%%%%%%%%

\title{Sensitive group actions on regular curves of almost $\leq n$ order}

\author{Suhua Wang \& Enhui Shi \& Hui Xu \& Zhiwen Xie}

\address[S.H. Wang]{Zhangjiagang Campus, Jiangsu University of Science and Technology, Zhangjiagang 215600, P. R. China}
\email{wangsuhuajust@163.com}

\address[E.H. Shi]{School of Mathematical Sciences, Soochow University, Suzhou 215006, P. R. China}
\email{ehshi@suda.edu.cn}

\address[H. Xu]{School of Mathematical Sciences, Soochow University, Suzhou 215006, P. R. China}
\email{20184007001@stu.suda.edu.cn}

\address[Z.W. Xie]{School of Mathematical Sciences, Soochow University, Suzhou 215006, P. R. China}
\email{2893591806@qq.com}

\begin{abstract}
Let $X$ be a regular curve and $n$ be a positive integer such that for every nonempty open set $U\subset X$, there is
a nonempty connected open set $V\subset U$ with the cardinality $|\partial_X(V)|\leq n$. We show that if $X$ admits a sensitive action of
a group $G$, then $G$ contains a free subsemigroup and the action has positive geometric entropy. As a corollary, $X$ admits no sensitive nilpotent group action.
\end{abstract}


\maketitle

\pagestyle{myheadings} \markboth{Wang-Shi-Xu-Xie }{Sensitive group actions}

\section{Introduction}

Sensitivity is usually regarded as an important character of chaotic systems. Many systems imply sensitivity, such as nontrivial
weak mixing systems, non-equicontinuous minimal systems, and topologically transitive systems with dense minimal points \cite{GW}.
In the definitions of Devaney's chaos and Ruelle-Taken's chaos (also called Auslander-Yorke's chaos),
the notion of sensitivity is a key ingredient \cite{AY,De}. One may consult \cite{AK,LTY,SY,Su,Ye} for some very interesting studies around
various forms of sensitivity.

The notion of expansivity is closely related to the theory of structure stability in differential dynamical systems,
which is stronger than sensitivity. It is well known that hyperbolic automorphisms on torus and full shifts on
symbolic spaces are expansive. Expansive algebraic actions by $\mathbb Z^d$ or amenable groups are also intensively studied \cite{CL, LS}.

Which space can admit an expansive $\mathbb Z$ action is an extensively studied question.
It is well known that the Cantor set, solenoids, and each compact orientable surface of positive genus admit expansive $\mathbb Z$ actions \cite{OR,Wi}.  However, the
circle $\mathbb S^1$ and the $2$-sphere $\mathbb S^2$ admit no expansive $\mathbb Z$ actions \cite{Hir}. A famous result due to Ma\~{n}\'{e} \cite{Ma} says
that if a compact metric space $X$ admit an expansive $\mathbb Z$ action, then $X$ is of finite dimension. One may refer
to \cite{Ka96,KM,Mo02,Mo09} for recent progresses on the study of existence of expansive $\mathbb Z$-actions on $1$-dimensional continua.

T. Ward once asked whether the circle $\mathbb S^1$ can admit an expansive nilpotent group action \cite{KLP}. This question was
solved negatively by C. Connell, A. Furman and S. Hurder in an unpublished paper \cite{CFH} using the ping pong game technique. More strong results by
G. Margulis can be found in \cite{Mar}. These stimulate the following
\begin{ques}\label{ques}
 Given a group $G$ and a continuum $X$, can $G$ act on $X$ expansively/sensitively?
\end{ques}
The answer towards this question certainly depends on the topology of the phase space and involves the algebraic structure of the acting group.
For abelian group $G$, some people showed the nonexistence of expansive
or sensitive $G$ actions on some locally connected continua \cite{MS07, MS12}. Relying on the semi ping pong technique as in \cite{CFH},
Shi and Wang showed the nonexistence of expansive nilpotent group actions on Peano continua with a free dendrite \cite{SW}.

The purpose of this paper is to study sensitive group actions on regular curves. Regular curves are a class of $1$-dimensional
continua, which are natural generalization of graphs and dendrites. Some well-known curves such as the triangular Sierpinski
curve and Apollonius curves are regular. We will introduce the notion of regular curve of almost $\leq n$ order (see Section 2 for
the definition), which covers many important continua including the mentioned as above. The dynamics of continuous maps or homeomorphisms on regular curves
were also studied very recently \cite{Ka06, Ka07, Nag17, Nag20}.

The following is the main theorem we obtained  in this paper. Noting that a regular curve of almost $\leq n$ order may have no free dendrites, such as the
triangular Sierpinski curve, the present theorem is not implied by the results in \cite{SW} even in the case of expansive actions.

\begin{thm}\label{Main Theorem}
Let $X$ be a regular curve of almost $\leq n$ order for some positive integer $n$. Let $G$ be a group
acting on $X$ sensitively. Then $G$ contains a free noncommutative subsemigroup and the action has
positive geometric entropy. In particular, $G$ cannot be nilpotent.
\end{thm}

The idea of the proof is to use ping pong game technique as in \cite{CFH}. However, since the topology of
regular curves is more complicated than that of the circle and sensitivity is weaker than expansivity,
we have to develop some new ideas to overcome these difficulties.
As a supplement to the main theorem, we remark that the solvable group  $\mathbb Z\ltimes \mathbb Z_2$ can act on the closed
interval $[0, 1]$ expansively \cite{SZ}.

One may expect that if a continuum $X$ admits no expansive $\mathbb Z$ actions, then it admits no expansive nilpotent
group actions. However, this is not true in general. According to Ma\~{n}\'{e}'s results in \cite{Ma}, if $X$ is of
infinite dimension, then it admits no expansive $\mathbb Z$ actions; but there do exist an expansive $\mathbb Z^2$ action
on the infinite product of circles $\mathbb T^\infty$ \cite{SZ}. Moreover,  Mouron constructed for each positive integer $n$,
a continuum $X$ which admits an expansive $\mathbb Z^{n+1}$ action but admits no expansive $\mathbb Z^n$ actions \cite{Mo10}. These examples
show the essential differences between expansive $\mathbb Z$ actions and expansive $\mathbb Z^n$ actions with $n\geq 2$.
Up to now, for nilpotent group $G$, the answer to Question \ref{ques} only concentrates on the case that $X$ is locally connected one dimensional continua.
Few is known when $X$ is not locally connected or is of dimension $>1$.

The main theorem of this paper is only a partial answer to the following question.

\begin{ques}\label{question2}
Do there exist  sensitive nilpotent group actions on regular curves?
\end{ques}

Here, we remark that  H. Kato proved that Suslinian continua admit no expansive homeomorphisms in \cite{Ka90}. This  implies that there are no expansive $\mathbb{Z}$ actions on regular curves as every regular curve is a Suslinian continuum. In addition, Mai and Shi constructed a sensitive homeomorphism on a Suslinian continuum in \cite{MS072}, which answered a
question proposed by Kato in \cite{Ka93}. So, the above question is not true if we replace ``regular curves" by ``Suslinian continua".
These motivate us to ask the following question.

\begin{ques}\label{question2}
Do there exist expansive nilpotent group actions on Suslinian continua?
\end{ques}

In Section 2, we will introduce some general notions and facts around group actions and the topology of continua.
We show a convergence property of subcontinua contained in regular curves in Section 3 and obtained in Section 4 a relation between
sensitivity and the existence of transitive open subsystems for group actions on regular curves. Based on these preparations,
we prove in the last section the existence of semi ping pong for any sensitive group actions on regular curves of almost
$\leq n$ order.

\section{Preliminaries}

\subsection{General notions around group actions}

Let $X$ be a compact metric space and let ${\rm Homeo}(X)$ be
the homeomorphism group of $X$. Suppose $G$ is a group. A group
homomorphism $\phi: G\rightarrow {\rm Homeo}(X)$ is called a {\it continuous
action} of $G$ on $X$; we use the symbol $(X, G, \phi)$ to denote this action.
 For brevity, we usually use $gx$ or $g(x)$ instead of $\phi(g)(x)$
and use $(X, G)$ instead of $(X, G, \phi)$ if no confusion occurs.

For a subset $A$ of $X$, we denote by $\overline{A}$ the closure of $A$ in $X$. The diameter of $A$ is defined by ${\rm diam}(A)={\rm sup}\{d(x,y):x,y\in A\}$, where $d$ is the metric on $X$. For $x\in X$,  the {\it orbit} of $x$  is the set $Gx\equiv\{gx:g\in
G\}$; $K\subset X$ is called {\it $G$-invariant} if $Gx\subset K $ for every $x\in K$.
If $K$ is $G$-invariant, we naturally get the {\it restriction action} of $G$ on $K$,
which is denoted by $(K, G|_K)$. An action $(X, G)$ is called {\it topologically  transitive} (or {\it transitive} for brevity) if for any nonempty open subsets $U$ and $V$ of $X$, there is some $g\in G$ such that $g(U)\cap V\neq\emptyset$; is called {\it point transitive}
if the closure ${\overline{Gx}}=X$ for some $x\in X$, and $x$ is called a {\it transitive point}. It is well known that when $X$ is a compact metric space
and $G$ is countable, the definitions of topological transitivity and point transitivity are equivalent. $(X,G)$ is called {\it sensitive} if for some $c>0$ and
for every nonempty open set $U$ in $X$ there is $g\in G$ such that the diameter ${\rm diam}(g(U))>c$, where $c$ is called a {\it sensitivity constant} of $(X,G)$.

\subsection{Geometric entropy and semi ping pong}

A group $G$ is called {\it nilpotent} if there is a normal series of subgroups $G=G_0\rhd G_1\rhd \cdots\rhd G_n=\{e\}$ with $G_{i+1}=[G_i,G]$ for each $i$. Nilpotent groups contain no free subsemigroups. Let $(X, G)$ be an action of group $G$ on space $X$. A tuple $(A, B_1, B_2, g_1, g_2)$ is called a {\it semi ping pong} if $A, B_1, B_2\subset X$ and $g_1, g_2\in G$ such that $B_1\subset A$, $B_2\subset A$, $B_1\cap B_2=\emptyset$,
$g_1(A)\subset B_1$, and $g_2(A)\subset B_2$ (see Fig. 1). It is well known and easy to check that if $(A, B_1, B_2, g_1, g_2)$ is
a semi ping pong, then the semigroup $S(g_1, g_2)$ generated by $g_1$ and $g_2$ is free. This implies immediately that
$G$ cannot be nilpotent. In some sense, the notion of semi ping pong can be viewed as a generalization of horseshoe for
group actions, the existence of which means a kind of complexity of the system.

The notion of geometric entropy was introduced
by Ghys, Langevin, and Walczak in \cite{GLW} for group actions, which is a generalization of topological entropy for systems generated by a single map. Let $G$ be a finitely generated group, and let $S$ be a symmetric generating set of $G$ such that the identity is contained in $S$ and $s\in S$ if and only if $s^{-1}\in S$. For an element $g\in G$, if there exist $s_1,s_2,\cdots, s_n\in S$ such that $g=s_1\cdots s_n$, then we say that $g$ has word length $||g||\leq n$. Let $(X,G,\phi)$ be a group action. Given $\epsilon>0$ and $n>0$, two distinct points $x,y\in X$ are called {\it $(n,\epsilon)$-separated} if there exists some $g\in G$ such that $||g||\leq n$ and $d(gx,gy)>\epsilon$. A finite subset of $X$ is called an {\it $(n,\epsilon)$-separated subset of $X$} if any two distinct points in the set are $(n,\epsilon)$-separated. Denote by $s(\phi, \epsilon,n)$ the maximum cardinality of all $(n,\epsilon)$-separated subsets of $X$. The geometric entropy $h(\phi)$ of a group action $(X,G,\phi)$ is defined as follows:
$$h(\phi)=\lim_{\epsilon\to 0} \left[\limsup_{n\to\infty}\ \frac{\log s(\phi,\epsilon,n)}{n}\right].$$

The positivity of geometric entropy  is also a description of the complexity of the action. The existence of semi ping pong implies
the positivity of geometric entropy  for group actions on compact metric spaces.

\begin{figure}%[htbp]
\centering
\includegraphics[scale=0.5]{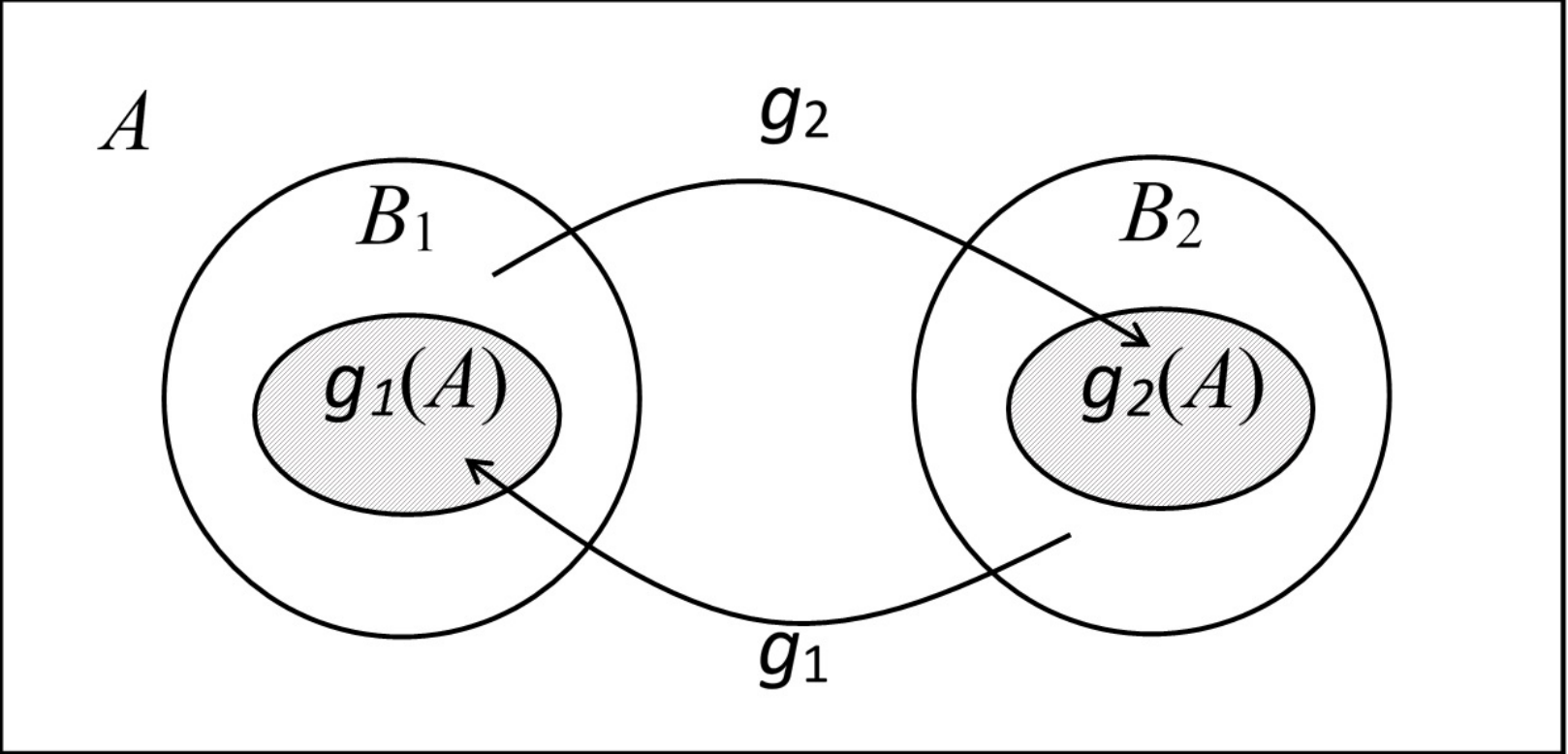}
\centerline{Fig.1. The semi ping pong}
\end{figure}

\subsection{Notions and results in continuum theory}

Recall that a {\it continuum} means a compact connected metric space.
A continuum $X$ is called a
{\it regular curve} if for every $x\in X$ and every open neighborhood $U$ of $x$, there is an open neighborhood
$V$ of $x$ contained in $U$ such that  $|\partial_X(V)|$ (the cardinality of the boundary $\partial_X(V)$ of $V$ in $X$) is finite; if the cardinality$|\partial_X(V)|\leq n$ for some
previously fixed positive integer $n$, then $X$ is called a {\it regular curve of order $\leq n$}.
 A regular curve $X$ is said to be of {\it almost $\leq n$ order} for some fixed positive
integer $n$, if for every nonempty open set $U$ of $X$, there is an open set $V\subset U$ such that $|\partial_X(V)|\leq n$.
Clearly, any regular curve of $\leq n$ order is almost $\leq n$ order by the definition.
It is easy to check that dendrites are regular curves of almost $\leq 2$ order; and the triangular Sierpinski
curve (see Fig. 2) is a regular curve of almost $\leq 3$ order. However, there are many dendrites not of finite order, such as
the infinite star (see Fig. 3) which is the union of countable many arcs with diameters tending to $0$ and with one endpoint being as a common intersecting point. Regular curves are known to be locally connected.

\begin{figure}%[htbp]
\centering
\includegraphics[scale=0.6]{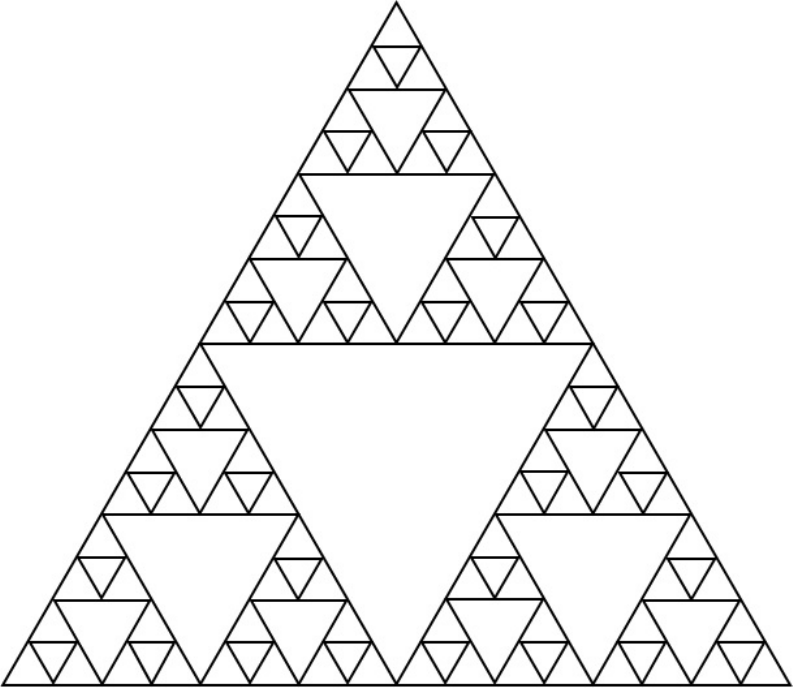} \hspace{1cm}
\includegraphics[scale=0.8]{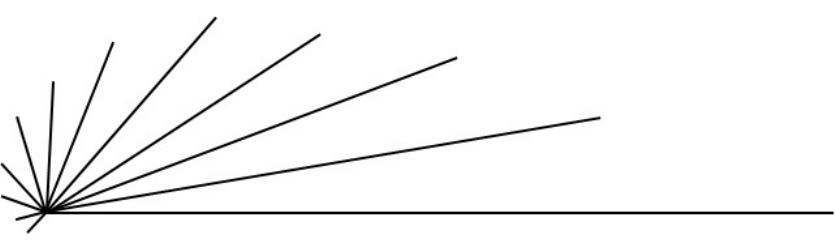}
\centerline{Fig. 2. The triangular Sierpinski curve \hspace{2.5cm} Fig. 3. The infinite star \hspace{2cm}}
\end{figure}

Let $(X,d)$ be a continuum. For every $x\in X$, $B(x,\epsilon)$ denotes the open ball with center $x$ and radius $\epsilon$.
Let $2^X$ be the family of all nonempty closed subsets of $X$, which is called the {\it hyperspace} of $X$. Let
  $C(X)$ be the set of all subcontinua of $X$. For each $\epsilon>0$ and $A\in 2^X$, let $N_d(\epsilon,A)=\{x\in X: d(x,a)<\epsilon \ {\rm for\ some\ } a\in A\}$. For $A,B\in 2^X$, define $d_H(A,B)={\rm inf}\{\epsilon>0: A\subset N_d(\epsilon,B) {\rm \ and\ } B\subset N_d(\epsilon,A)\}$. Then $d_H$ is a metric on  $2^X$ and is called the {\it Hausdorff metric} on it, with respect to which
 $2^X$ and $C(X)$ are  compact metric spaces (see e.g. \cite{Na}).

The following theorem is known as the Boundary Bumping Theorem (see e.g. \cite[p. 73]{Na}) which will be used in the sequel.

\begin{thm}\label{Boundary bumping}
Let $X$ be a continuum and let $U$ be a nonempty proper open subset of $X$. If $K$ is a component of $\overline U$, then $K\cap \partial_X(U)\not=\emptyset.$
\end{thm}

\section{Topology of regular curves}

\begin{lem}\label{large subcontinua}
Let $X$ be a regular curve and $U$ be a connected open subset of $X$ with $|\partial_X(U)|=n$ for some positive integer $n$.
If ${\rm diam}(U)>\epsilon$ for some $\epsilon>0$, then there is some connected open set $V\subset U$, such that $d(V, \partial_X(U))\geq \epsilon/4n$
and ${\rm diam}(V)\geq \epsilon/4n$.
\end{lem}

\begin{proof}
Let $\partial_X(U)=\{e_1, e_2, \cdots, e_n\}$. We claim that $U\setminus\bigcup_{i=1}^{n}B(e_i,\epsilon/2n)\neq\emptyset$. In fact, if $U\subset\bigcup_{i=1}^{n}B(e_i,\epsilon/2n)$, then for any two distinct points $a,b\in U$, by the connectivity of $U$ there are finite $B(e_{i_1},\epsilon/2n), \cdots, B(e_{i_m},\epsilon/2n)$ ($m\leq n$) such that $a\in B(e_{i_1},\epsilon/2n)$, $b\in B(e_{i_m},\epsilon/2n)$ and $B(e_{i_k},\epsilon/2n)\cap B(e_{i_{k+1}},\epsilon/2n)\neq\emptyset$ for $1\leq k\leq m-1$. Thus we have
\begin{align*}
d(a,b)\leq & d(a,e_{i_1})+d(e_{i_1},e_{i_2})+\cdots +d(e_{i_{m-1}},e_{i_m})+d(e_{i_m},b)\\
< & \frac{\epsilon}{2n}+\frac{\epsilon}{n}\cdot(m-1)+\frac{\epsilon}{2n}\\
\leq & \epsilon.
\end{align*}
It follows that ${\rm diam}(U)\leq\epsilon$, which is a contradiction. Hence $U\setminus\bigcup_{i=1}^{n}B(e_i,\epsilon/2n)\neq\emptyset$.

Take a point $x\in U\setminus\bigcup_{i=1}^{n}B(e_i,\epsilon/2n)$. Then $d(x,\partial_X(U))\geq \epsilon/2n$, and hence
$$d(B(x,\epsilon/4n), \partial_X(U))\geq \epsilon/4n.$$
 Let $V$ be the component of $U\cap B(x,\epsilon/4n)$ which contains $x$. Then $d(V,\partial_X(U))\geq\epsilon/4n$. Let $W=U\cap B(x,\epsilon/4n).$
Since $U$ is connected, $\emptyset\not= \partial_X(W)\subset\partial_X(B(x,\epsilon/4n))$.
This together with Theorem \ref{Boundary bumping} implies $\emptyset\not= \partial_X(V)\subset \partial_X(B(x,\epsilon/4n))$.
 So, ${\rm diam}(V)={\rm diam}(\overline{V})\geq \epsilon/4n$.
\end{proof}

%The following Lemma comes from \cite[pp.60-61]{Na}.
%\begin{lem}\label{sequence converge}
%Let $X$ be a compact metric space. Then every sequence of compact subsets (resp. subcontinua) of $X$ has a subsequence converging to a compact subset (resp. subcontinuum) of %$X$.
%\end{lem}

\begin{prop}\label{convergence continua}
Let $X$ be a regular curve and let $(U_i)_{i=1}^\infty$ be a sequence of connected open subsets of $X$ with $|\partial_X(U_i)|=n$ for some positive integer $n$ and for each $i$. Suppose that there is some $\epsilon>0$ with ${\rm diam}(U_i)>\epsilon$ for each $i$. Then there are a nonempty open subset $W$ of $X$ and infinitely many $i$'s such that $W$ is contained in $U_i$.
\end{prop}

\begin{proof}
For each $i$, it follows from Lemma~\ref{large subcontinua} that there is a connected open subset $W_i\subset U_i$ with $d(W_i, \partial_X(U_i))\geq \epsilon/4n$ and ${\rm diam}(W_i)\geq \epsilon/4n$. By the compactness of $2^X$ and $C(X)$, there are subsequences $(\overline{W_{i_k}})$ and $(\partial_X(U_{i_k}))$ such that $(\overline{W_{i_k}})$ converges to a subcontinuum $A$, and
$$d_H(\partial_X(U_{i_{k_1}}), \partial_X(U_{i_{k_2}}))<\epsilon/4n, \forall k_1\not=k_2.\eqno(3.1)$$
Take a point $z\in A$. Then there exists a connected open neighborhood $Q$ of $z$ such that $\partial_X(Q)$ is finite and ${\rm diam}(Q)<\epsilon/4n$. Since $(\overline{W_{i_k}})$ converges to $A$, there exists a positive integer $N$ such that $W_{i_k}\cap Q\neq\emptyset$ for each $k\geq N$. Noting that  ${\rm diam}(W_{i_k})\geq \epsilon/4n$ and ${\rm diam}(Q)<\epsilon/4n$, we have $W_{i_k}\nsubseteq Q$. Thus $W_{i_k}\cap \partial_X(Q)\neq\emptyset$ by the connectivity of $W_{i_k}$.
Hence, there exist a point $p\in \partial_X(Q)$ and infinitely many $k$'s such that $p\in W_{i_k}$. Passing to a subsequence if necessary, we may suppose that $p\in W_{i_k}$ for each $k\geq N$.

Let $W=W_{i_N}$. To complete the proof, we only need to show that $W\subset U_{i_k}$ for all $k\geq N$. Otherwise, there is some $k'\geq N$ with $W\nsubseteq U_{i_{k'}}$.  Since $W$ is connected and $p\in W\cap U_{i_{k'}}$, there is a point $e\in W\cap\partial_X(U_{i_{k'}})$.  By $(3.1)$,  there is $e'\in \partial_X(U_{i_N})$ such that $d(e,e')<\epsilon/4n$. Then we have $d(W, e')\leq d(e,e')<\epsilon/4n$. This contradicts the assumption that $d(W,\partial_X(U_{i_N}))\geq\epsilon/4n$ at the beginning.
\end{proof}

\section{Sensitivity and transitive open subsystem}

\begin{lem}\label{regular curve}
Let $X$ be a regular curve. Then for every constant $c>0$ there is a finite set $A\subset X$ such that the diameter of each component of $X\setminus A$ is not greater than $c$.
\end{lem}

\begin{proof}
Since $X$ is a regular curve, for every $x\in X$, there is an open neighborhood $V_x$ of $x$ such that $\partial_X (V_x)$ is finite and ${\rm diam}(V_x)\leq c$. Then the family $\mathcal{V}=\{V_x : x\in X\}$ is an open cover of $X$. By the compactness of $X$, there is a finite subcover $\{V_{x_i}\}_{i=1}^{n}\subset \mathcal{V}$ of $X$. Thus $\bigcup_{i=1}^{n}\partial_X(V_{x_i})$ is finite, which is denoted by $A$. For each component $C$ of $X\setminus A$, there exists some $V_{x_i}$ such that $C\cap V_{x_i}\neq\emptyset$. Since $C$ is connected and $C\cap \partial_X(V_{x_i})=\emptyset$, $C\subset V_{x_i}$. This implies that ${\rm diam}(C)\leq c$.
\end{proof}

The following proposition discusses the relations between sensitivity and the existence of transitive open subsystems for group actions on regular curves.

\begin{prop}\label{transitive open set}
Let $X$ be a regular curve and let $G$ be a group acting on $X$ sensitively. Then there is a $G$-invariant
open set $V$ in $X$ such that the restriction $(V, G|_V)$ is transitive.
\end{prop}

\begin{proof}
Let $c$ be a sensitivity constant of the action $(X,G)$. By Lemma~\ref{regular curve}, there is a finite set $A\subset X$ such that each component of $X\setminus A$
has diameter $\leq c$. We claim that $\overline{G(A)}=\bigcup_{a\in A}\overline{Ga}$ has non-empty interior. Otherwise,  $\overline{G(A)}$ is nowhere dense in $X$. Thus $X\setminus \overline{G(A)}$ is a $G$-invariant dense open subset of $X$. Noting that $X$ is locally connected, each component $U$ of $X\setminus \overline{G(A)}$ is
open in $X$, and for each $g\in G$, $g(U)\subset X\setminus\overline{G(A)}$. This implies that ${\rm diam}(g(U))\leq c$, which contradicts
the assumption that $c$ is a sensitivity constant of $(X,G)$. Thus the claim holds.

By the claim, there is an $a\in A$ such that ${\rm Int}(\overline{Ga})\neq\emptyset$. Let $V={\rm Int}(\overline{Ga})$. Then $V$ is a $G$-invariant open set and $(V,G|_V)$ is transitive.
\end{proof}

\section{Existence of semi ping pong}

Now let us prove the main theorem \ref{Main Theorem}. According to the discussions in Section 2.2, we need only to show the existence of semi ping pong.\vspace{2mm}

%\noindent{\bf Theorem 1.1.} {\it Let $X$ be a regular curve of almost $\leq n$ order for some positive integer $n$. Let $G$ be a group
%acting on $X$ sensitively. Then $G$ contains a free noncommutative subsemigroup and the action has
%positive geometric entropy. In particular, $G$ cannot be nilpotent.}

\begin{proof}
 By Proposition~\ref{transitive open set}, there is a $G$-invariant open subset $V\subset X$ such that $(V,G|_V)$ is transitive. Let $x_0$ be a transitive point of $(V,G|_V)$. Noting that $X$ is a regular curve of almost $\leq n$ order, for each  $m\in \mathbb{N}$, there is a nonempty open set $V_m\subset B(x_0,1/m)\cap V$ with $|\partial_X(V_m)|\leq n$. By sensitivity, there exists some $g_m\in G$ such that ${\rm diam}(g_m(V_m))>c$, where $c$ is a sensitivity constant of $(X,G)$. Clearly, $|\partial_X(g_m(V_m))|\leq n$ for every $m\in \mathbb{N}$.

Applying Proposition~\ref{convergence continua} to $(g_m (V_m))_{m=1}^{\infty}$, we get a subsequence $(g_{m_k}(V_{m_k}))_{k=1}^{\infty}$ and a nonempty open subset $D\subset X$ such that $D\subset \bigcap g_{m_k}(V_{m_k})$. Since $g_{m_k}^{-1}(D)\subset V_{m_k}\subset B(x_0, 1/m_k)$ for each $k$,  we have
$$\lim_{k\to\infty}g_{m_k}^{-1}(D)=\{x_0\}.\eqno(5.1)$$

Choose two nonempty open subsets $B_1, B_2\subset D$ with $B_1\cap B_2=\emptyset$. Since $x_0$ is a transitive point of $(V,G|_V)$, there are $f,h\in G$ such that $f(x_0)\in B_1$ and $h(x_0)\in B_2$. According to $(5.1)$ and the continuity of $f$ and $h$, there are $k_1$ and $k_2$ such that $fg_{m_{k_1}}^{-1}(D)\subset B_1$ and $hg_{m_{k_2}}^{-1}(D)\subset B_2$. Thus $\{D,B_1,B_2,fg_{m_{k_1}}^{-1},hg_{m_{k_2}}^{-1}\}$ becomes a semi ping pong. Thus we complete the proof.
\end{proof}

\subsection*{Acknowledgements}
The work is supported by NSFC (No. 11771318, 11790274).

%\newpage

\end{document}